\documentclass[reqno]{amsart}

\usepackage{amssymb,amsfonts, amsmath, amsthm, graphicx, enumerate, upgreek}
\usepackage{mathtools, etoolbox, bm}
\usepackage[table]{xcolor}
\usepackage{tikz}
\usepackage[all,arc]{xy}
\usepackage[utf8]{inputenc}
\usepackage[english]{babel}

\usepackage[hidelinks]{hyperref}
\newtheorem{thm}{Theorem}[section]
\newtheorem*{thm*}{Theorem}

\newtheorem{cor}[thm]{Corollary}
\newtheorem*{cor*}{Corollary}
\newtheorem{defthm}[thm]{Definition-Theorem} 
\newtheorem{prop}[thm]{Proposition}
\newtheorem*{prop*}{Proposition}
\newtheorem{lem}[thm]{Lemma}
\newtheorem*{lem*}{Lemma}

\newtheorem*{conj*}{Conjecture}

\newtheorem*{quest*}{Question}

\theoremstyle{definition}
\newtheorem{defn}[thm]{Definition}

\theoremstyle{remark}
\newtheorem{rem}[thm]{Remark}

\newtheorem*{ack*}{Acknowledgements}

\newcommand{\stmodsf}{\underline{\mathsf{mod}}\ \! } % stable module category
\newcommand{\modsf}{\mathsf{mod}\ \!} % finitely generated module category
\newcommand{\Modsf}{\mathsf{Mod}\ \!} % Module category
\newcommand{\N}{\mathbb{N}} % natural numbers
\newcommand{\m}{\mathfrak{m}} % maximal ideal
\newcommand{\depth}{{\rm depth}\ \!} % depth
\newcommand{\dell}{{\rm dell}\ \!} % delooping level 
\newcommand{\pdim}{{\rm pdim}\ \!} % projective dimension
\newcommand{\findim}{{\rm findim}\ \!} % little finitistic dimension
\newcommand{\Findim}{{\rm Findim}\ \!} % little finitistic dimension
\title[The finitistic dimension for algebras with radical square zero]{The finitistic dimension of an Artin algebra with radical square zero} 
\author{Vincent G\'elinas}
\address{vincent.gelinas@mail.utoronto.ca} 
\date{}

\begin{document}

\begin{abstract} We investigate the inequality $\Findim \Lambda^{op} \leq \dell \Lambda$ between the finitistic dimension and the delooping level of an Artin algebra $\Lambda$, and whether equality holds in general. We prove that equality $\Findim \Lambda^{op} = \dell \Lambda$ always holds for Artin algebras with radical square zero. 
\end{abstract}

\maketitle 
%\tableofcontents 
\setlength{\parskip}{5pt}

\section*{Introduction} 
The little finitistic dimension $\findim \Lambda$ and big finitistic dimension $\Findim \Lambda$ are important homological invariants of Artin algebras $\Lambda$. However, understanding their values in terms of more basic invariants of $\Lambda$ remains difficult. 

In the setting of commutative local Noetherian rings $R = (R, \m, k)$, the analogous question has a satisfactory answer: the little finitistic dimension $\findim R = \depth R$ is equal to the depth of $R$ by the Auslander--Buchsbaum formula, and the big finitistic dimension $\Findim R = \dim R$ equals its Krull dimension by theorems of Bass \cite{Ba62} and Gruson--Raynaud \cite{GR71}. 

The depth of $R$ can be defined in terms of the Ext functor via ${\rm Ext}^*_R(k, R)$, and so admits a natural extension to the setting of Noetherian semiperfect rings $\Lambda$, and in particular to that of Artin algebras (\cite[Section 1]{Ge20}). In this situation, the Auslander--Buchsbaum equality is replaced by an inequality due to Jans \cite{Ja61} 
\begin{align*} 
	\depth \Lambda \leq \findim \Lambda^{op} 
\end{align*}
between the (right) depth of $\Lambda$ and its (left) little finitistic dimension. However, equality fails to hold in general. In \cite{Ge20}, the author introduced a new invariant $\dell \Lambda$ for Noetherian semiperfect rings $\Lambda$, taking value in $\N \cup \{ \infty \}$, which refines Jans' inequality to 
\begin{align*}
	\depth \Lambda \leq \findim \Lambda^{op} \leq \dell \Lambda
\end{align*}
and gives rise for Artin algebras $\Lambda$ to finer inequalities  
\begin{align*}
	\depth \Lambda \leq \findim \Lambda^{op} \leq \Findim \Lambda^{op} \leq \dell \Lambda.
\end{align*}

This new invariant was shown to recover the depth for commutative local Noetherian rings $R$ in that $\depth R = \dell R$, so that the above collapse to equalities in the commutative case. Similar results hold for Noetherian semiperfect rings $\Lambda$ satisfying certain cohomological conditions (see \cite[Thm. 2.12]{Ge20}). However, the above inequalities are strict in general.  

In this paper, we restrict ourselves to the setting of Artin algebras $\Lambda$. Of the three inequalities 
\begin{align*}
	\depth \Lambda \leq \findim \Lambda^{op} \leq \Findim \Lambda^{op} \leq \dell \Lambda.
\end{align*}

the first two $\depth \Lambda \leq \findim \Lambda^{op}$ and $\findim \Lambda^{op} \leq \Findim \Lambda^{op}$ are known to be strict in some examples; however, the last inequality 
\begin{align*}
	\Findim \Lambda^{op} \leq \dell \Lambda 
\end{align*}
is less understood. In \cite{Ge20}, it was shown to be a form of the ``Cohen-Macaulay'' property. The following curious result was observed: 

\begin{thm*}[{\cite[Thm. 4.1]{Ge20}}] Let $\Lambda$ be a Noetherian semiperfect ring, and consider the inequality $\Findim \Lambda^{op} \leq \dell \Lambda$. 
	\begin{enumerate}[i)] 
	\item The inequality characterises Cohen-Macaulay rings amongst commutative local Noetherian rings, in which case we have equality.
	\item The inequality holds for all Gorenstein rings, in which case we have equality. 
	\item The inequality holds for Artinian rings. 
\end{enumerate}
\end{thm*}
In the first two cases, the inequality is actually an equality: when $R = (R, \m, k)$ is commutative local Noetherian the inequality boils down to $\dim R \leq \depth R$, which immediately forces $\dim R = \depth R$. For noncommutative Gorenstein rings $\Lambda$, the equality $\Findim \Lambda^{op} = \dell \Lambda$ is a consequence of powerful results of Angeleri-H\"ugel--Herbera--Trlifaj \cite{AHHT06}. This leaves the case of Artinian rings open, and leads to the natural question: 

\begin{quest*}[{\cite[Q.4.2]{Ge20}}] Does $\Findim \Lambda^{op} = \dell \Lambda$ hold for all Artin algebras $\Lambda$? 
\end{quest*} 

In this paper, we show that this question has a positive answer when $\Lambda$ has radical square zero. Our main result is: 

\begin{thm*}[Thm. \ref{maintheorem}] Let $\Lambda$ be an Artin algebra with radical square zero. Then
	\begin{align*}
		\Findim \Lambda^{op} = \dell \Lambda.	 
	\end{align*} 
\end{thm*}

We additionally give a combinatorial formula for $\dell \Lambda$, and therefore for $\Findim \Lambda^{op}$, in terms of the valued quiver of $\Lambda$ (Proposition \ref{PropB}). This theorem gives some hope that the answer to Question 4.2 should be positive in general.

\subsection*{Conventions and notation} By an Artin algebra $\Lambda$, we will always mean an $R$-algebra $\Lambda$ which is module finite over a commutative Artinian ring $R$. We let $\Modsf \Lambda$ stand for the category of right $\Lambda$-modules, and $\modsf \Lambda$ for the full subcategory of finitely presented right modules. We identify left modules with right modules over the opposite ring $\Lambda^{op}$. Without further description, the term module always refers to a finitely presented right $\Lambda$-module. We let $D: \modsf \Lambda \to \modsf \Lambda^{op}$ stand for the duality between right and left $\Lambda$-modules, see \cite[II.3]{ARS95}.  

We denote by $\stmodsf \Lambda$ the stable module category of $\Lambda$, whose objects agree with those of $\modsf \Lambda$ and whose morphisms are module homomorphism modulo homomorphisms factoring through projectives. We refer to the retracts and isomorphisms in $\stmodsf \Lambda$ as \emph{stable retracts} and \emph{stable equivalences}, respectively. It is well known that, for $X, Y \in \modsf \Lambda$, the module $X$ is a stable retract of $Y$ in $\stmodsf \Lambda$ if and only if $X$ is a retract of $Y \oplus P$ in $\modsf \Lambda$ for some projective $P$, and moreover $X, Y$ are stably equivalent if and only if $X \oplus Q \cong Y \oplus P$ in $\modsf \Lambda$ for some projectives $P, Q$ (see e.g. \cite[Lemma 0.1]{Ge20}).

\section{The main invariants} 
Let $\Lambda$ be an Artin algebra with radical ${\sf r}$ satisfying ${\sf r}^2 = 0$. From now on $\Lambda$ will always denote such an algebra with radical square zero, unless specified otherwise. We write $\{T_1, T_2, \dots, T_n \}$ for a complete set of simple left modules in $\modsf \Lambda^{op}$, and similarly $\{S_1, S_2, \dots, S_n \}$ for a complete set of simple right modules in $\modsf \Lambda$. The little and big finitistic dimensions of $\Lambda$ are defined respectively by
\begin{align*}
	\findim \Lambda &= \sup \{ \pdim M \ | \ M \in \modsf \Lambda \textup{ and } \pdim M < \infty \}	\\
	\Findim \Lambda &= \sup \{ \pdim M \ | \ M \in \Modsf \Lambda \textup{ and } \pdim M < \infty \}. 
\end{align*}
More precisely these are the right finitistic dimensions of $\Lambda$, and the left finitistic dimensions are given by $\findim \Lambda^{op}$ and $\Findim \Lambda^{op}$. We will be particularly interested in the latter. 

We next define the invariant $\dell \Lambda$, called ``delooping level'' in \cite{Ge20}. The reader is referred to \cite[Section 1]{Ge20} for more details. First, we let $$\Omega: \stmodsf \Lambda \to \stmodsf \Lambda$$ denote the syzygy functor on the stable module category, which can computed on objects as the kernel $\Omega M = {\rm ker}(P \twoheadrightarrow M)$ of an epimorphism from a projective. Let $$\Sigma: \stmodsf \Lambda \to \stmodsf \Lambda$$ be the left adjoint to $\Omega$, which was introduced by Auslander--Reiten \cite{AR96}. Since the pair $(\Sigma, \Omega)$ consists of adjoint endofunctors, taking powers gives rise to a series of adjoint pairs $(\Sigma^n, \Omega^n)$ for each $n \geq 0$. Each module $M \in \stmodsf \Lambda$ then admits a unit and counit map
\begin{align*}
	\eta_n& :M \to \Omega^n \Sigma^n M\\
	\varepsilon_n& : \Sigma^n \Omega^n M \to M. 
\end{align*}

To define $\dell \Lambda$, we first define an invariant $\dell S$ for simple modules $S \in \modsf \Lambda$ taking value in $\N \cup \{ \infty \}$. 

\begin{defthm}[{\cite[Def. 1.2, Thm. 1.10]{Ge20}}] The following are equivalent for $\dell S$: 
	\begin{enumerate}[a)] 
	\item $\dell S \leq n$. 
	\item $\Omega^n S$ is a stable retract of $\Omega^{n+1} N$ for some $N \in \modsf \Lambda$. 
	\item $\Omega^n S$ is a stable retract of $\Omega^{n+1} \Sigma^{n+1} \Omega^n S$. 
	\item The unit map $\eta_{n+1}: \Omega^n S \to \Omega^{n+1} \Sigma^{n+1} \Omega^n S$ is a split-monomorphism. 
\end{enumerate}
\end{defthm} 
\begin{rem} In this article we shall focus on condition b) as a mean of computing $\dell S$. However, the conditions c)-d) are of great practical and theoretical interest in general, and so we include them as part of the definition. 
\end{rem} 

We then define $\dell \Lambda$ by taking supremum over the simples $S$ in $\modsf \Lambda$. 
\begin{defn} The delooping level of $\Lambda$ is defined as $\dell \Lambda := \sup_S \dell S$. 
\end{defn}
Unpacking the definition, we see that $\dell \Lambda \leq n$ if and only if $\Omega^n S$ is a stable retract of $\Omega^{n+1}N$ for some $N \in \modsf \Lambda$ for each simple $S \in \modsf \Lambda$.

Finally, we note that $\Findim \Lambda^{op} \leq \dell \Lambda$ for any Artin algebra (\cite[Prop. 1.3]{Ge20}). To investigate whether this is an equivalence in our setting, we first find methods to calculate each side.  

\section{The finitistic dimension of radical square zero algebras} We can compute the finitistic dimensions of $\Lambda^{op}$ in terms of the projective dimensions of left simple modules. Define the invariant $s$ as 
\begin{align*}
	s = \sup \{ \pdim T \ | \ T \in \modsf \Lambda^{op} \textup{ simple with } \pdim T < \infty \}. 
\end{align*}

Mochizuki \cite{Mo65} showed the following (see also \cite[Section 3]{HZ95}). 
\begin{prop}[\cite{Mo65}] We have (in)equalities 
	\begin{align*}
		s \leq \findim \Lambda^{op} = \Findim \Lambda^{op} \leq s+1.
	\end{align*}
\end{prop}
Since the little and big finitistic dimensions coincide in our situation, in the sequel we only focus on $\Findim \Lambda^{op}$. We can decide whether $\Findim \Lambda^{op} = s$ or $\Findim \Lambda^{op} = s+1$ as follows. 

\begin{lem} The following are equivalent: 
	\begin{enumerate}[i)]
	\item $\Findim \Lambda^{op} = s$.
	\item Every left simple $T$ with $\pdim T = s$ satisfies $T^* = 0$. 
\end{enumerate}
\end{lem}
\begin{proof}
$i) \implies ii)$. Let $T$ be a simple module with $\pdim T = s$. If $T^* \neq 0$, then there exists a non-zero morphism $\varphi: T \to$ $\! _\Lambda \Lambda$ which is then an embedding as $T$ is simple. Letting ${\rm coker}(\varphi) = N$ gives $\Omega N = T$ and $\pdim N = s+1$, contradicting i). 

$ii) \implies i)$. Assume that $\Findim \Lambda^{op} \neq s$, so that $\Findim \Lambda^{op} = s+1$. Let $M$ be a left module with $\pdim M = s+1$. Let $\pi: P \to M$ be its projective cover. Then $\Omega M = {\rm ker}(\pi) \subseteq {\sf r}P$ is annihilated by ${\sf r}$ (as ${\sf r}^2 = 0$) and so is semisimple. Since $\pdim \Omega M = s$, there exists a simple summand $T$ of $\Omega M$ with $\pdim T = s$. But then $T \subseteq P$ implies $T^* \neq 0$, contradicting ii). 
\end{proof}

\section{The delooping level of radical square zero algebras} 
Next, we consider the computation of $\dell \Lambda$; we will obtain a combinatorial formula for this invariant in terms of the valued quiver of $\Lambda$. We will then be able to match the combinatorial formula to the description of $\Findim \Lambda^{op}$ above. 

Recall from \cite[III.1]{ARS95} that the valued quiver $\Gamma = (\Gamma_0, \Gamma_1)$ of $\Lambda$ consists of a set of vertices $\Gamma_0 = \{ 1, 2, \dots, n \}$, and $\Gamma_1$ consists of a set of arrows between vertices $(i, j)$, precisely one arrow $\alpha: j \to i$ for each pair such that ${\rm Ext}^1_\Lambda(S_i, S_j) \neq 0$. The endomorphism ring of each simple $\Delta_i = {\rm End}_\Lambda(S_i)$ is a division algebra, which we use to equip each arrow $j \to i$ with a valuation, namely a pair of integers $(m_i, m_j)$ defined as 
\begin{align*}
	&m_i = \dim_{\Delta_i}{\rm Ext}^1_\Lambda(S_i, S_j) \\
	&m_j = \dim_{\Delta_j}{\rm Ext}^1_\Lambda(S_i, S_j).
\end{align*} 
Note that using the duality $D: \modsf \Lambda \to \modsf \Lambda^{op}$, one sees that the valued quiver of $\Lambda^{op}$ is the opposite quiver $\Gamma^{op}$ obtained by reversing arrows and valuations $(m_i, m_j) \leftrightarrow (m_j, m_i)$. When $\Lambda = kQ/{\sf r}^2$ is a quiver path algebra over a field $k$, the valued quiver $\Gamma$ has the same vertices as $Q$ and has, for every pair of vertices $(i, j)$ with $m$ arrows $j \to i$, a single arrow with valuation $(m, m)$. 

In general we can use the valuation quiver to encode the resolutions of simple modules. 
\begin{prop}[{\cite[Prop. 1.15 a)]{ARS95}}] Let $\Lambda$ be an Artin algebra with radical ${\sf r}$. Let $P_i \to S_i$ and $P_j \to S_j$ be the projective covers of the simple modules $S_i$ and $S_j$ respectively. Then the following numbers are the same.  \begin{enumerate}[i)] 
		\item $m_j = \dim_{\Delta_j}{\rm Ext}^1_\Lambda(S_i, S_j)$. 
		\item The multiplicity of the simple module $S_j$ as a summand of $P_i{\sf r}/P_i{\sf r}^2$. 
\end{enumerate}
\end{prop}

Since in our situation ${\sf r}^2 = 0$, this provides a description of the syzygies of simple modules. 

\begin{cor}\label{firstsyzygy} For each simple $S_i \in \modsf \Lambda$ we have a short exact sequence 
	\begin{align*}
		0 \to \bigoplus_{j \to i} S_j^{m_j} \to P_i \to S_i \to 0 
	\end{align*} 
	where the sum is over all arrows $j \to i$ in $\Gamma$. 
\end{cor}
\begin{proof}
We have a short exact sequence
\begin{align*}
	0 \to P_i{\sf r} \to P_i \to S_i \to 0 
\end{align*}
and $P_i{\sf r} = \bigoplus_{j \to i} S_j^{m_j}$ by the last proposition since ${\sf r}^2 = 0$. 
\end{proof}

We now have a description of the first syzygies of $S_i$ as
\begin{align*}
	\Omega S_i = \bigoplus_{j \to i} S_j^{m_j}.	 
\end{align*}
Iterating Corollary \ref{firstsyzygy}, we obtain a description of the higher syzygies of simple modules. 

\begin{cor}\label{iteratedsyzygy} The higher syzygies of simples $S_i \in \modsf \Lambda$ are given by
	\begin{align*}
	\Omega^n S_i = \bigoplus_{j \to j_1 \to \dots \to j_{n-1} \to i} S_j^m
	\end{align*}
	where the sum is over all paths of length $n$ from $j$ to $i$ in $\Gamma$ and the multiplicity is $m := m_j m_{j_1} \dots m_{j_{n-1}}$. 
\end{cor}

We next aim to describe the delooping level of the simples $S_i$. Recall that this is given by 
\begin{align*}
	\dell S_i = \inf \{ n \geq 0 \ | \ \Omega^n S_i \textup{ is a stable retract of } \Omega^{n+1} N \textup{ for some } N \in \modsf \Lambda \}.	
\end{align*}

We have that $\dell S_i = 0$ if and only if $S_i$ is a retract of a syzygy module, if and only if $S_i$ embeds in a projective. Since $S_i$ is simple, this is equivalent to $S_i^* \neq 0$. We can give a combinatorial characterisation of this last condition. Recall that a vertex $i \in \Gamma_0$ is a source if there are no arrows $j \to i$, and is a sink if there are no arrows $i \to j$. 

\begin{lem}\label{projectiveinjective} Let $i \in \Gamma_0$ be a vertex. 
	\begin{enumerate}[a)] 
	\item $i$ is a sink in $\Gamma$ if and only if $S_i$ is injective.
	\item $i$ is a source in $\Gamma$ if and only if $S_i$ is projective. 
\end{enumerate}
\end{lem}
\begin{proof}
By definition $i$ is a sink if and only if there are no arrows $i \to j$ in $\Gamma$, if and only if ${\rm Ext}^1_\Lambda(S_j, S_i) = 0$ for all simples $S_j$, if and only if $S_i$ is injective. This proves a), and the proof of b) is dual.  
\end{proof}

\begin{lem}\label{LemmaA} The following are equivalent for a vertex $i \in \Gamma_0$: 
	\begin{enumerate}[a)] 
	\item $\dell S_i = 0$. 
	\item $S_i^* \neq 0$. 
	\item $i$ is a source or not a sink. 
\end{enumerate}
\end{lem} 
\begin{proof}
We saw the equivalence $a) \iff b)$ in the paragraph above Lemma \ref{projectiveinjective}. 

$b) \implies c)$: Let $\varphi: S_i \to \Lambda_\Lambda$ be a non-zero morphism. Since $S_i$ is simple, $\varphi$ is an embedding of $S_i$ into $\Lambda_\Lambda$. Now assume that $i$ is a sink. By the Lemma \ref{projectiveinjective}, the simple $S_i$ is injective, and therefore the short exact sequence
\begin{align*}
	0 \to S_i \xrightarrow{\varphi} \Lambda_\Lambda \to {\rm coker}\ \! \varphi \to 0
\end{align*}
splits. This shows that $S_i$ is projective and therefore $i$ is a source by Lemma \ref{projectiveinjective}. Hence we have shown that if $i$ is a sink, then it must be a source, and c) holds.  

$c) \implies b)$. If $i$ is a source then $S_i$ is projective and so $S_i^* \neq 0$. If $i$ is not a sink, then there is an arrow $i \to j$ in $\Gamma$ and therefore ${\rm Ext}^1_\Lambda(S_j, S_i) \neq 0$. By Corollary \ref{firstsyzygy}, the simple $S_i$ is a summand of $\Omega S_j$. Since $\Omega S_j$ embeds in a projective, so does $S_i$ and thus $S_i^* \neq 0$. Hence in either situation b) holds.  
\end{proof} 

For the higher case, we have the next proposition.
\begin{prop}\label{PropB} The following are equivalent for the vertex $i \in \Gamma_0$: 
	\begin{enumerate}[a)]
	\item $\dell S_i \leq n$. 
	\item For every path of length $n$ 
		$$j \to j_1 \to \dots \to j_{n-1} \to i$$
		either $j$ is a source, or there exists a path of length $n$ 
		$$j \to j'_1 \to \dots \to j'_{n-1} \to k$$ 
		with $k$ not a sink. 
\end{enumerate}
\end{prop}
\begin{proof} Lemma \ref{LemmaA} already covers the case $n = 0$ and so we may assume that $n \geq 1$. 

$a) \implies b)$. Assuming $\dell S_i \leq n$, then $\Omega^n S_i$ is a stable retract of $\Omega^{n+1} N$ for some $N \in \modsf \Lambda$. Since ${\sf r}^2 = 0$, the first syzygy $\Omega N$ is semisimple, say with decomposition $\Omega N = \bigoplus_{k} S_k$. Each such simple $S_k$ then satisfies $S_k^* \neq 0$, and so $k$ is either a source (if $S_k$ is projective) or not a sink. We then have 
\begin{align*} 
	\Omega^{n+1} N = \bigoplus_{k} \Omega^n S_k. 
\end{align*}
Moreover as $n \geq 1$, we may ignore those $k$ in the decomposition above with $S_k$ projective and therefore assume that for each $S_k$, the vertex $k$ is not a sink. Expanding this out using Corollary \ref{iteratedsyzygy}, we obtain a decomposition
\begin{align*}
	\Omega^{n+1} N = \bigoplus_k \bigoplus_{j' \to j'_1 \to \dots \to j'_{n-1} \to k} S_{j'}^{m'}
\end{align*} 
running over all paths of length $n$ in $\Gamma$, and with $m'$ the multiplicity. Similarly we have 
\begin{align*}
	\Omega^n S_i = \bigoplus_{j \to j_1 \to \dots \to j_{n-1} \to i} S_j^m.	
\end{align*}
Since $\Omega^n S_i$ is a stable retract of $\Omega^{n+1} N$, there exists a projective $P$ such that $\Omega^n S_i$ is a module retract of $P \oplus \Omega^{n+1} N$. By the Krull-Schmidt theorem, it follows that every summand $S_j$ of $\Omega^n S_i$ is either projective, in which case $j$ is a source, or occurs as $S_j = S_{j'}$ in the decomposition of $\Omega^{n+1}N$ above, in which case there is a path
\begin{align*}
	j = j' \to j'_{1} \to \dots \to j'_{n-1} \to k	
\end{align*}
with $k$ not a sink. Hence condition b) holds. 

$b) \implies a)$. We run the argument in reverse. For each simple summand $S_j$ of $\Omega^n S_i$, we see that either $j$ is a source, or there is a path of length $n$ in $\Gamma$ 
\begin{align*} 
	j \to j'_{1} \to \dots \to j'_{n-1} \to k 	
\end{align*}
with $k$ not a sink. In the first case $S_j$ is projective, and in the second case $S_j$ is a summand of $\Omega^n S_k$ with $S_k$ a simple satisfying $S_k^* \neq 0$. But then there is an embedding $\varphi_k: S_k \hookrightarrow \Lambda_\Lambda$, and letting $N_k = {\rm coker}\ \! \varphi_k$, we obtain $S_k = \Omega N_k$.

Letting $P = \bigoplus_j S_j$ be the sum of those $S_j$ which were projective (in the first case) and $N = \bigoplus_k N_k$ (from the second case), we obtain that $\Omega^n S_i$ is a retract of $P \oplus \Omega^{n+1} N$. This proves that $\dell S_i \leq n$ and condition a) holds. 
\end{proof}

The previous proposition gives a combinatorial formula for $\dell S_i$ for each simple $S_i$, and therefore also a formula for $\dell \Lambda = \sup_i \dell S_i$. 

\section{Proof of the main theorem} 
To compare $\dell \Lambda$ to $\Findim \Lambda^{op}$, recall that we have defined the invariant
\begin{align*} 
s = \sup \{ \pdim T \ | \ T \in \modsf \Lambda^{op} \textup{ simple with } \pdim T < \infty \} 
\end{align*}
and that we always have $s \leq \Findim \Lambda^{op} \leq s+1$. We can recast the value $s$ combinatorially in terms of the valued quiver of $\Lambda$ as follows. First, given a vertex $j \in \Gamma_0$, denote by $\downarrow{j} \subseteq \Gamma$ the full subquiver consisting of vertices $i$ reachable by a path $j \to \dots \to i$. When $\downarrow{j} \subseteq \Gamma$ is a subquiver without oriented cycle, we denote by $\ell(\downarrow{j})$ the length of the longest path in it starting from $j$. 

\begin{lem}\label{LemmaC} We have $s = \sup \{ \ell(\downarrow{j}) \ | \ \textup{the subquiver} \downarrow{j} \subseteq \Gamma \textup{ has no oriented cycle} \}$. 
\end{lem} 
\begin{proof}
	To each vertex $j \in \Gamma_0$ we let $T_j \in \modsf \Lambda^{op}$ be the associated left simple module. Applying Corollary \ref{iteratedsyzygy} to the simple $T_j$ we see that simple summands of $n$-th syzygies of $T_j$ are given by 
	\begin{align*}
		\Omega^n T_j = \bigoplus_{i \to i_1 \to \dots \to i_{n-1} \to j} T_i^m. 
	\end{align*}
	Hence the simples $T_i$ occuring as summands of syzygies of $T_j$ are those whose vertex $i$ can reach $j$ by a path in the valuation quiver $\Gamma^{op}$ of $\Lambda^{op}$. Since the latter reverses orientations, these are precisely the vertices $i \in \downarrow{j} \subseteq \Gamma$. 

	The above description of the syzygies then shows that $\pdim T_j < \infty$ if and only $\downarrow{j}$ has no oriented cycle, in which case $\pdim T_j = \ell(\downarrow{j})$. The formula then holds. 
\end{proof}

We can now refine the inequality $s \leq \Findim \Lambda^{op} \leq s+1$. 
\begin{prop} We have $s \leq \Findim \Lambda^{op} \leq \dell \Lambda \leq s+1$. 
\end{prop}
\begin{proof} We always have $\Findim \Lambda^{op} \leq \dell \Lambda$ and so it suffices to show $\dell \Lambda \leq s+1$. 

	By contradiction, assume that $\dell \Lambda > s+1$, so that $\dell S_i > s+1$ for some right simple $S_i$. By Proposition \ref{PropB}, we conclude that there exists a path of length $s+1$ in $\Gamma$ 
	\begin{align*}
		j \to j_1 \to \dots \to j_{s-1} \to j_s \to i	
	\end{align*}
	such that $j$ is not a source, and such that every path of length $s+1$ 
	\begin{align*} 
		j \to j'_1 \to \dots \to j'_{s-1} \to j'_s \to k
	\end{align*}
	ends at a sink vertex $k$. 

	By the second condition, we conclude that the subquiver $\downarrow{j} \subseteq \Gamma$ has no oriented cycle since it contains no path of length $\geq s+2$. Since there is a path of length $s+1$ from $j$ to $i$, we conclude that $\ell(\downarrow{j}) = s+1$. But this contradicts maximality of $s$ in Lemma \ref{LemmaC}. 
\end{proof} 

Finally, we arrive at the proof of the main theorem.
\begin{thm}\label{maintheorem} Let $\Lambda$ be an Artin algebra with radical square zero. Then we have $$\Findim \Lambda^{op} = \dell \Lambda.$$ 
\end{thm}
\begin{proof} If $\dell \Lambda = s$ the inequality $s \leq \Findim \Lambda^{op} \leq \dell \Lambda \leq s+1$ gives the result, and so we may assume that $\dell \Lambda = s+1$. We will show that $\Findim \Lambda^{op} = s+1$ using a similar argument to the previous proof. 

	Since $\dell \Lambda > s$, by Proposition \ref{PropB} we see that there exists a sink $i \in \Gamma_0$ and a path of length $s$ in $\Gamma$ 
\begin{align*}
	j \to j_1 \to \dots \to j_{s-1} \to i	
\end{align*}
such that $j$ is not a source in $\Gamma$, and for which every path of length $s$ 
\begin{align*}
	j \to j'_1 \to \dots \to j'_{s-1} \to k
\end{align*}
ends in a sink vertex $k$. As seen above, this implies that $\pdim T_j = \ell(\downarrow{j}) = s$. 

We now claim that $T_j^* \neq 0$. Note that the condition that $j$ is not a source in $\Gamma$ means that $j$ is not a sink in $\Gamma^{op}$. Applying Lemma \ref{LemmaA} to $\Lambda^{op}$, we obtain that $T_j^* \neq 0$ as claimed. Finally, since $T_j$ is simple there is an embedding $\varphi: T_j \hookrightarrow \Lambda^{op}$, and letting $N = {\rm coker}(\varphi)$, we obtain $\Omega N = T_j$ and so that $\pdim N = s+1$. This shows that $s+1 \leq \Findim \Lambda^{op} \leq \dell \Lambda = s+1$, and so $\Findim \Lambda^{op} = \dell \Lambda$. 
\end{proof}

\bibliographystyle{alpha} 
\bibliography{rad.bib} 

\end{document}